\documentclass[a4paper,11pt]{article}
\usepackage{graphicx}
\usepackage{amssymb}
\usepackage{enumerate}
\usepackage{array}
\usepackage{geometry}
\usepackage{fancyhdr}
\usepackage{fixltx2e}
\usepackage{amsmath}
\usepackage{verbatim}
\usepackage{url}

\DeclareMathOperator{\torord}{Tord}
\DeclareMathOperator{\tor}{Tor}

\newtheorem{theorem}{Theorem}[section]
\newtheorem{lemma}[theorem]{Lemma}
\newtheorem{proposition}[theorem]{Proposition}

\newtheorem{defn}[theorem]{Definition}

\newenvironment{proof}[1][Proof]{\begin{trivlist}
\item[\hskip \labelsep {\bfseries #1}]}{\end{trivlist}}

\newenvironment{remark}[1][Remark]{\begin{trivlist}
\item[\hskip \labelsep {\bfseries #1}]}{\end{trivlist}}

\begin{document}

\title{Preserving torsion orders when embedding into\\groups with `small' finite presentations}
\date{\today}
\author{Maurice Chiodo,  Michael E. Hill}

\maketitle

\begin{abstract}
We give a complete survey of a construction by Boone and Collins \cite{ref:26_relations} for embedding any finitely presented group into one with $8$ generators and $26$ relations. We show that this embedding preserves the set of orders of torsion elements, and in particular torsion-freeness. We combine this with a result of Belegradek \cite{ref:uni_tor_free_2} and Chiodo \cite{ref:uni_tor_free_1} to prove that there is an $8$-generator $26$-relator \emph{universal} finitely presented torsion-free group (one into which all finitely presented torsion-free groups embed).
\end{abstract}

\let\thefootnote\relax\footnotetext{2010 \textit{AMS Classification:} 20E06, 20F05.}
\let\thefootnote\relax
\let\thefootnote\relax\footnotetext{The first author undertook this work as part of a project that has received funding from the European Union's Horizon 2020 research and innovation programme under the Marie Sk\l{}odowska-Curie grant agreement No.~659102.}
\let\thefootnote\relax\footnotetext{The second author was partially supported the Bridgwater Summer Research Fund, and the Beker Mathematics Fund.}

\section{Introduction}\label{sec:intro}

One famous consequence of the Higman Embedding Theorem \cite{ref:het_higm} is the fact that there is a \emph{universal} finitely presented (f.p.)~group; that is, a finitely presented group into which all finitely presented groups embed. Later work was done to give upper bounds for the number of generators and relations required. In \cite{ref:42_relations} Valiev gave a construction for embedding any given f.p.~group into a group with $14$ generators and $42$ relations.  In \cite{ref:26_relations} Boone and Collins improved this to $8$ generators and $26$ relations. Later in \cite{ref:21_relations} Valiev further showed that $21$ relations was possible. In particular we can apply any of these constructions to a universal f.p.~group to get a new one with only a few generators and relations.

More recently Belegradek in \cite{ref:uni_tor_free_2} and Chiodo in \cite{ref:uni_tor_free_1} independently showed that there exists a universal f.p.~torsion-free group; that is, an f.p.~torsion-free group into which all f.p.~torsion-free groups embed. This can be generalised even further: if ${X \subseteq \mathbb{N}}$ then we say a group $G$ is \emph{$X$-torsion-free} if whenever ${g^n = 1}$ with ${n \in X}$ then ${g = 1}$; in \cite[Theorem 2.11]{ref:uni_X-tor_free} Chiodo and McKenzie showed that there is a universal $X$-torsion-free f.p.~group for \emph{any} recursively enumerable set $X$.

In this paper we show that the embedding construction of Boone and Collins in \cite{ref:26_relations} preserves \emph{torsion orders}, the set of orders of torsion elements of a group $G$ (denoted $\torord(G)$), to show the following result.\\

\noindent \textbf{Theorem~\ref{thm:main_result}\:} \textit{There is a uniform algorithm for embedding a finitely presented group $A$ into another finitely presented group $H$ with $8$ generators and $26$ relations. Moreover, ${\torord(A) = \torord(H)}$.} \\

We then apply this to the main result in \cite[Theorem 2.11]{ref:uni_X-tor_free} to show the following: \\

\noindent \textbf{Theorem~\ref{thm:low_relator_uni_group}\:} \textit{Let $X$ be a recursively enumerable set. Then there is a universal finitely presented $X$-torsion-free group with $8$ generators and $26$ relations.} \\

Taking $X=\mathbb{N}$ in the above theorem gives a universal f.p.~torsion-free group with $8$ generators and $26$ relations. Furthermore, in Theorem~\ref{thm:lowish_relator_uni_group} we write down an explicit presentation of a universal f.p.~$X$-torsion-free group with $13$ generators and $33$ relations (with the exception of one extremely long relation, which we do not write down explicitly as it is complicated and depends on finding the group in \cite[Theorem 2.11]{ref:uni_X-tor_free} explicitly).

\section{Overview}\label{sec:overview}

This paper follows the construction given by Boone and Collins in \cite{ref:26_relations} for embedding any f.p.~group into one with $8$ generators and $26$ relations. Unfortunately, their construction in \cite{ref:26_relations} is not completely self-contained. In particular, they make use of several other embedding constructions, from several other papers (\cite{ref:het_aand, ref:group_to_semigroup, ref:borisov_orig, ref:borisov_col, ref:semigroup_to_group}). The objective of this paper is thus twofold:

\begin{enumerate}
 \item To give a complete survey of the construction by Boone and Collins in \cite{ref:26_relations}, and all peripheral constructions (as well as precisely where to find them), for embedding any finitely presented group into one with $8$ generators and $26$ relations.
 \item To note the torsion preserving properties of these constructions, where necessary, to prove our main results in Theorems \ref{thm:main_result} and \ref{thm:low_relator_uni_group}.
\end{enumerate}

Thus, this paper contains a complete description of a construction for embedding an arbitrary finitely presented group $A$ into a finitely presented group $H$ with $8$ generators and $26$ relations following the exposition in \cite{ref:26_relations}, and a proof that ${\torord(A) = \torord(H)}$. To succeed with objective 2 we were, in a sense, forced to carry out the survey in objective 1. It was only through a \emph{careful} analysis of \cite{ref:26_relations}, including all peripheral constructions, that we were able to complete objective 2. We now give the structure of the rest of our paper.

In Section~\ref{sec:prelims} we give the necessary preliminaries and background results needed to proceed with the constructions and proofs. These are all quite standard.

In Section~\ref{sec:group_to_semigroup} we give an overview of the construction by Boone, Collins and Matijasevi\v{c} in \cite{ref:group_to_semigroup} for simulating equality in any f.p.~group $A$ in a corresponding semigroup $\mathfrak{N}$ with only $2$ generators and $3$ relations. However, one of these relations is extremely long, which we discuss.

In Section~\ref{sec:semigroup_to_group} we give an overview of the construction by Collins in \cite{ref:semigroup_to_group}  for simulating equality in the above semigroup $\mathfrak{N}$ in a group $G$ with $7$ generators and $14$ relations. The purpose of $G$ is to `simulate' (but not embed) our original group $A$ in a $7$-generator $14$-relation group $G$. We note that $G$ is built up by several HNN-extensions, starting with the free group on $2$ generators, and hence is torsion-free.

In Section~\ref{sec:simulation} we give an overview of the construction by Boone and Collins in \cite{ref:26_relations} (which is a variant of Aanderaa's proof of the Higman Embedding Theorem in \cite{ref:het_aand}) to obtain an explicit embedding of our original group $A$ into a group $K_{4}$ with $13$ generators and $33$ relations. Since this is built up entirely by HNN-extensions, starting with the group $A*G$, we see that this embedding preserves torsion orders.

In Section~\ref{sec:explict} we write down a (mostly) explicit presentation for the group $K_{4}$ constructed in Section~\ref{sec:simulation}, except for one particular `long' relation, which we discuss. This long relation is inherited from the long relation in $\mathfrak{N}$ from Section~\ref{sec:semigroup_to_group}.

In Section~\ref{sec:borisov} we give an overview of the construction by Collins in \cite{ref:borisov_col} (which is itself a refinement of an idea due to Borisov in \cite{ref:borisov_orig}) for embedding an f.p.~group with many commuting generators into an f.p.~group with fewer generators and relations. We note that this embedding preserves torsion orders as it is built from HNN-extensions and an amalgamated free product.

Finally, in Section~\ref{sec:final} we apply the embedding construction from Section~\ref{sec:borisov} to the finite presentation $K_{4}$ given in Section~\ref{sec:explict} (as is done by Boone and Collins in \cite{ref:26_relations}), which we know to have several commuting relations. We end up with a presentation with $8$ generators and $26$ relations, though we do not write it out explicitly as it is very complicated. Thus, we obtain a torsion order preserving embedding of our original group $A$ into a group with $8$ generators and $26$ relations.

\section{Preliminaries}\label{sec:prelims}

This paper relies heavily on studying torsion in groups, thus it is necessary to introduce some notation on this. So, let $G$ be any group. We define the set of \emph{torsion elements} of $G$ to be $\tor(G) := \left\lbrace g \in G \:|\: \exists n>1 ,\: g^n=1 \right\rbrace$, and the set of \emph{torsion orders} of $G$ to be $\torord(G) := \left\lbrace n \in \mathbb{N} \:|\: n>1,\: \exists g \in G ,\: o(g)=n \right\rbrace$. $G$ is said to be \emph{torsion-free} if $\torord(G)=\emptyset$. 

It is possible to generalise the notion of torsion, which we do as follows: given any ${X \subseteq \mathbb{N}}$, we define the \emph{$X$-torsion} of $G$ to be $\tor^X(G) := \left\lbrace g \in G \:|\: \exists n \in X ,\: n>1 ,\: g^n=1 \right\rbrace$. For any set  ${X \subseteq \mathbb{N}}$, we define the \emph{factor completion} of $X$ to be $X\textsuperscript{fc} := \{ n \in \mathbb{N} \:|\: \exists m \geq 1,$ $nm \in X \} $, and we say $X$ is \emph{factor complete} if ${X\textsuperscript{fc} = X}$. $G$ is said to be \emph{$X$-torsion-free} if $\torord(G)\cap X\textsuperscript{fc}=\emptyset$; equivalently, if $\tor^{X}(G)=\{e\}$. We note that, in most cases, the $X$-torsion of a group behaves almost identically to the torsion of a group, when considering embedding theorems.

Crucial to our work are the standard embedding constructions of \emph{HNN extensions} and \emph{amalgamated free products}, which we define here for the convenience of the reader.

Let $G$ be a group and let $A_i$, $B_i$ be subgroups of $G$ for ${1 \leq i \leq n}$. Let ${\phi_i : A_i \rightarrow B_i}$ be isomorphisms for ${1 \leq i \leq n}$. We define the \emph{HNN-extension} of $G$ over $\phi_1, \ldots ,\phi_n$ to be
\begin{equation*}
G*_{\phi_1,\ldots,\phi_n} := G * \left\langle t_1, \ldots ,t_n \right\rangle / \left\langle\left\langle \left\lbrace \phi_i(g)^{-1}t_i^{-1}gt_i \:|\: g \in A_i,\: 1 \leq i \leq n \right\rbrace \right\rangle \right\rangle^{G * \left\langle t_1, \ldots ,t_n \right\rangle}
\end{equation*}

Let $G$ and $H$ be groups with $A$ a subgroup of $G$ and $B$ a subgroup of $H$. Let ${\phi : A \rightarrow B}$ be an isomorphism. We define the \emph{amalgamated free product} of $G$ and $H$ over $\phi$ to be  
\begin{equation*}
G*_{\phi}H := G * H / \left\langle \left\langle \left\lbrace \phi(g)^{-1}g \:|\: g \in A \right\rbrace \right\rangle \right\rangle^{G * H}
\end{equation*}

We will often be considering presentations of HNN-extensions and amalgamated free products. These can always be written by adding some generators and relations to a presentation of the original base group(s). Thus, it is convenient to introduce the following notation: For a group $G$ and $R \subseteq G * \left\langle t_1, \ldots ,t_n \right\rangle$ we define
\begin{equation*}
\left\langle G;\: t_1,\: \ldots t_n \:|\: R \right\rangle := (G * \left\langle t_1, \ldots ,t_n \right\rangle) / \left\langle\left\langle R \right\rangle \right\rangle^{G * \left\langle t_1, \ldots ,t_n \right\rangle}
\end{equation*}

It turns out that the torsion in an HNN-extension or amalgamated free product is extremely well-understood, as given by the following two standard results. We will make repeated use of these in our analysis of the embedding theorems later in this paper.

\begin{proposition}\emph{\cite[Theorem~11.69]{ref:rotman}}\label{prop:hnn_torsion}
Let ${H := G*_{\phi_1,\ldots,\phi_n}}$ be a $HNN$-extension of $G$. Then every torsion element of $H$ is conjugate to an element of $G$. In particular, ${\torord(G) = \torord(H)}$.
\end{proposition}

\begin{proposition}\emph{\cite[Theorem~11.69]{ref:rotman}}\label{prop:amlg_torsion}
Let ${K := G*_{\phi}H}$ be an amalgamated free product. Then every torsion element of $K$ is conjugate to an element of $G$ or $H$. In particular, $\torord(K) =$ $\torord(G) \cup \torord(H)$.
\end{proposition}

A famous embedding theorem, originally due to Higman, Neumann and Neumann in \cite{ref:HNN}, allows us to embed any countable group into a 2-generator group. Moreover, by the above, the torsion in such an embedding is well understood. We state this here, in the form needed for our later analysis:

\begin{theorem}\label{thm:two_generator}
Let $G$ be a finitely presented group. There is a uniform algorithm for embedding $G$ into a finitely presented group $G'$ with $2$ generators, with ${\torord(G) = \torord(G')}$. 
\end{theorem}
\begin{proof}
This is a standard construction which can be found in \cite[Theorem 11.71]{ref:rotman}. The observation that this preserves torsion  orders can be seen in  \cite[Lemma 2.16]{ref:tor_len}.
\end{proof}

We will frequently need to deal with arbitrary presentations of groups and semigroups. These may have complicated relations, and are often difficult to work with. However, given any finite presentation of a semigroup (or indeed a group), there is a useful trick to re-write the presentation with only `short' relators, each of length at most $3$. The trade-off is that this process introduces many more generators and relators, but for our purposes it is worthwhile. We state the precise result here; note that the same process works for group presentations, but we only require it for semigroups.

\begin{proposition}\label{prop:length_3}
Let ${A = \left\langle a_1, \ldots , a_k \:|\: R_i = S_i \right\rangle}$ be a finite semigroup presentation where $R_i$, $S_i$ are non-trivial for each $i$. Then we can algorithmically add new generators to $A$ to get a finite presentation of the form ${A = \left\langle a_i \:|\: A_i = B_i \right\rangle}$ where each $A_i$ is a single letter and each $B_i$ has length $2$.
\end{proposition}
\begin{proof}
Write the first relation as ${x_1 \ldots x_n = y_1 \ldots y_m}$. Define $a_{k+1} := x_1x_2$, $a_{k+i} :=$  $a_{k+i-1}x_{i+1}$ for ${1 < i \leq n}$,  ${a_{k+n+1} := y_1y_2}$ and ${a_{k+n+i} := a_{k+n+i-1}y_{i+1}}$ for ${1 < i < m}$. Then we can add ${a_{k+1}, \ldots ,a_{k+m+n-1}}$ to the generating set along with their defining relations and the relation ${a_{k+n} = a_{k+n+m-1}y_n}$ to the presentation and then the initial relation is redundant. Now repeat for the rest of the relations. 
\end{proof}

A set ${X \subseteq \mathbb{N}}$ is said to be \emph{recursively enumerable} (r.e.) if there is a Turing machine, whose inputs are binary representations of integers, with halting set $X$. 

Given any word ${W = x_1^{a_1} \cdots x_n^{a_n}}$ we will define $W^{\#} := x_1^{-a_1} \cdots x_n^{-a_n}$.

We finish this section by stating a crucial result, which is the motivation for this paper, and which we explicity appeal to at the conclusion of our analysis.

\begin{theorem}\label{thm:uni_X-tor_free}
Let $X$ be a recursively enumerable set. Then there is a universal finitely presented $X$-torsion-free group $G$. i.e., $G$ is $X$-torsion-free, and for any finitely presented (even recursively presented) group $H$ is $X$-torsion-free we have an embedding $H \hookrightarrow G$.
\end{theorem}
\begin{proof}
This was shown in \cite[Theorem 2.11]{ref:uni_X-tor_free}. A special case where $X = \mathbb{N}$ can found in either \cite{ref:uni_tor_free_2} or \cite{ref:uni_tor_free_1}.
\end{proof}

\section{Simulating an arbitrary f.p.~group $A$ in a semigroup $\mathfrak{N}$}\label{sec:group_to_semigroup}

In this section we give an overview of the construction by Boone, Collins and Matijasevi\v{c} in \cite{ref:group_to_semigroup} for simulating equality in any f.p.~group $A$ in a corresponding semigroup $\mathfrak{N}$ with only $2$ generators and $3$ relations. We follow the exposition given in p.$1$--$2$ of Section~$1$ of \cite{ref:26_relations}.

Let $A$ be a finitely presented group. As noted in Theorem~\ref{thm:two_generator} we can uniformly embed $A$ in a $2$ generated group in a way which preserves torsion orders. So w.l.o.g.~we may assume that ${A = \left\langle a_1, a_2 | R_i \right\rangle}$ has $2$ generators. We now will construct a semigroup $\mathfrak{N}$ and map on words ${\chi : A \rightarrow \mathfrak{N}}$ as in p.$1$--$2$ of \cite{ref:26_relations}. We note that this construction was initially given in \cite{ref:group_to_semigroup}.

Define extra generators ${a_3 := a_1^{-1}}$ and ${a_4 := a_2^{-1}}$ so that each $R_i$ is now a positive word on $a_1$, $a_2$, $a_3$, $a_4$. Now define the semigroup
\begin{equation*}
A_* := \left\langle a_1,\: a_2,\: a_3,\: a_4,\: p \:|\: R_ip = p,\: a_jp = pa_j \:\: (\forall i,j) \right\rangle
\end{equation*}
Let ${a_5 := p}$. Since none of the defining relations have the empty word on either side we can use Proposition~\ref{prop:length_3} to algorithmically add extra generators so that each relation of $A_*$ has the form ${a_\lambda = a_\mu a_\nu}$. So we have a presentation for $A_{*}$ of the following form:
\begin{equation*}
A_* = \left\langle a_1, \ldots , a_r \:|\: A_i = B_i \right\rangle
\end{equation*}
where each $A_{i}$ has length $1$, and each $B_{i}$ has length 2. By repeating relations we may assume that there are ${s = 2^t}$ relations for some ${t \in \mathbb{N}}$.

We now define the map $\psi$ on words over $\{a_{1}, \ldots, a_{r}\}$ induced by ${\psi(a_i) := \beta^2\gamma^i\beta\gamma^{r+1-i}}$. Define the semigroup
\begin{equation*}
\mathfrak{L} := \left\langle \beta,\: \gamma \:|\: \psi(A_j) = \psi(B_j) \right\rangle
\end{equation*}
Then the induced map ${\overline{\psi} : A_* \rightarrow \mathfrak{L}}$ is a semigroup homomorphism. Observe that $\psi(a_i)$ has length ${u=r+4}$ for each $i$; so $\psi(A_j)$ has length $u$ and $\psi(B_j)$ has length $2u$. Let $x_{i,j}$ be the $i$\textsuperscript{th} letter of $\psi(A_j)$ and let $y_{i,j}$ be the $i$\textsuperscript{th} letter of $\psi(B_j)$. We define the `interlacing words'
\begin{align*}
M &:= x_{1,1} x_{1,2} \cdots x_{1,s} x_{2,1} \cdots x_{u,s} \\ 
N &:= y_{1,1} y_{1,2} \cdots y_{1,s} y_{2,1} \cdots y_{2u,s} 
\end{align*}

Now define the semigroup $\mathfrak{M}$ with presentation
\begin{equation*}
\mathfrak{M} := \left\langle \beta,\: \gamma,\: \varepsilon \:|\: \varepsilon\beta\beta = \beta,\: \varepsilon\gamma\beta = \beta,\: \varepsilon\beta\gamma = \gamma,\: \varepsilon\gamma\gamma = \gamma,\: M = N \right\rangle
\end{equation*}
and a  map $\tau$ on words over $\{\alpha, \beta, \epsilon\}$ induced by ${\tau(\beta) := \sigma\alpha}$, ${\tau(\gamma) := \sigma}$, ${\tau(\epsilon) := \alpha^2}$. From this, we define the semigroup
\begin{equation*}
\mathfrak{N} := \left\langle \alpha,\: \sigma \:|\: \sigma\alpha^2 = \alpha^2\sigma\alpha\sigma,\: \sigma\alpha = \alpha^2\sigma^2,\: \tau(M) = \tau(N) \right\rangle
\end{equation*}
Then it is easily verified that the induced map ${\overline{\tau} : \mathfrak{M} \rightarrow \mathfrak{N}}$ is a semigroup homomorphism. Finally, we let $\chi$ be the composition of $\psi$ and $\tau$ (as word maps) from words over $\{a_{1}, \ldots, a_{r}\}$ to words over $\{\alpha, \sigma\}$. Note that this is well defined on words but is not necessarily a semigroup homomorphism. However, as noted in Theorem~$1$ of \cite{ref:26_relations}, we have the following properties.

\begin{theorem}\emph{\cite[Theorem~$1$]{ref:26_relations}}\label{thm:semigroup_properties}
Let ${\Phi_0 := \chi(p)\sigma\alpha^{2t}}$ and let $U$ and $V$ be any words in $A$. Then
\begin{enumerate}[(i)]
\item \label{prt:multiplicative} ${\chi(UV) \equiv \chi(U)\chi(V)}$ as words.
\item \label{prt:simulation} ${U = V}$ in $A$ iff ${\chi(U)\Phi_0 = \chi(V)\Phi_0}$ in ${\mathfrak{N}}$.
\item \label{prt:sortofsurjectivity} Moreover for any word $Z$ of $\mathfrak{N}$ if ${Z\Phi_0 = \chi(V)\Phi_0}$ in ${\mathfrak{N}}$ then there exists a word $W$ in $A$ such that ${W = V}$ in $A$ and ${\chi(W) \equiv Z}$ as words.
\end{enumerate}
\end{theorem}
\begin{proof}
(\ref{prt:multiplicative}) follows immediately from the definition of $\chi$. Proof of (\ref{prt:simulation}) and (\ref{prt:sortofsurjectivity}) can be found in Lemmata~$2.15$ and~$2.16$ of \cite{ref:group_to_semigroup}.
\end{proof}

Thus by Theorem~\ref{thm:semigroup_properties}~(\ref{prt:simulation}) we see that we can `simulate' equality in $A$ by equality of corresponding words in ${\mathfrak{N}}$.

\section{Constructing an f.p.~group $G$ simulating the semigroup $\mathfrak{N}$}\label{sec:semigroup_to_group}

In this section we give an overview of the construction by Collins on p.$306$--$308$ in Part~I of \cite{ref:semigroup_to_group}, for simulating equality in the semigroup $\mathfrak{N}$ in a group $G$ with $7$ generators and $14$ relations. The purpose of $G$ is to `simulate' (but not embed) our original group $A$ in a $7$-generator $14$-relation group. $G$ will be built up by several HNN-extensions, starting with the free group on $2$ generators, and thus be torsion-free.

We start with $\mathfrak{N}$ from Section \ref{sec:group_to_semigroup}, and for convenience we identify $\alpha$ with $s_1$ and $\sigma$ with $s_2$. We will now define the group ${G = G(\mathfrak{N}, \Phi_0)}$ as in Part I of \cite{ref:semigroup_to_group}. Note that, compared to the original construction in \cite{ref:semigroup_to_group}, $\mathfrak{N}$ will replace $\mathfrak{L}$, and we will write $W^{\#}$ instead of $\overline{W}$. First we make the following definition:
\begin{defn}\label{def:augmented_ssemigroup}
We define the semigroup
\begin{equation*}
\mathfrak{N}_* := \left\langle s_1,\: s_2,\: q \:|\: F_mq = qK_m \:\: (m=1,\ldots,5) \right\rangle
\end{equation*}
where
\begin{align*}
F_1 &= s_2s_1^2 & K_1 &= s_1^2s_2s_1s_2 \\
F_2 &= s_2s_1^2 & K_2 &= s_1^2s_2^2 \\
F_3 &= \tau(M)  & K_3 &= \tau(N) \\
F_4 &= s_1      & K_4 &= s_1 \\
F_5 &= s_2      & K_5 &= s_2
\end{align*}
\end{defn}

Observe that the $F_{m}$'s (resp.~$K_{m}$'s) are just the left (resp.~right) sides of the relations of $\mathfrak{N}$, re-written in terms of $s_{1}, s_{2}$.

\begin{defn}\label{def:H}
We define the following sequence of groups
\begin{align*}
F := H_1 &:= \left\langle a,\: d | - \right\rangle \\
     H_2 &:= \left\langle H_{1};\: s_1,\: s_2 \:|\: s_b^{-1}as_b = a,\: s_b^{-1}ds_b = d^6ad^6 \:\: (b=1,2) \right\rangle \\
     H_3 &:= \left\langle H_2;\: q \:|\: q^{-1}(d^mad^mF_m^{\#})q = K_md^mad^m \:\: (m = 1,\ldots ,5) \right\rangle \\
     H_4 &:= \left\langle H_3;\: t \:|\: t^{-1}at = a,\: t^{-1}dt = d \right\rangle \\
G := H_5 &:= \left\langle H_4;\: k \:|\: k^{-1}ak=a,\: k^{-1}dk = d,\: k^{-1}(\Phi_0^{-1}q^{-1}tq\Phi_0)k = (\Phi_0^{-1}q^{-1}tq\Phi_0) \right\rangle
\end{align*}
where $\Phi_0 := \chi(p)s_2s_1^{2t} \ (= s_2s_1s_2s_1s_2^6s_1s_2^{r-3}s_1^{2t})$.
\end{defn}

\begin{remark}
In \cite{ref:semigroup_to_group} $H_2$ is called $G_4$, $H_3$ is called $G_2$ and $H_4$ is called $G_1$ (with $G_3$ absent.)
\end{remark}

\begin{lemma}\label{lem:H_are_HNN}
$H_{i+1}$ is a HNN-extension of $H_i$ for each $i=1,2,3,4$.
\end{lemma}
\begin{proof}
$H_5$ and $H_4$ are HNN-extensions of $H_4$ and $H_3$ with the identity maps on the subgroups ${\left\langle a, d, \Phi_0^{-1}q^{-1}tq\Phi_0 \right\rangle}$ and ${\left\langle a, d \right\rangle}$  respectively. Observe that ${\left\langle a, d^6ad^6 \right\rangle}$ is the free group on two generators as a word reduced on $a$ and $d^6ad^6$ is reduced on $a$ and $d$ as well. Hence $H_2$ is an HNN-extension of $H_1$ along the isomorphism sending ${a \mapsto a}$, ${d \mapsto d^6ad^6}$. Finally it can be shown (such as in Lemma~$15$ in Part I of \cite{ref:semigroup_to_group}) that the subgroups ${\left\langle d^mad^mF_m^{\#},\ m=1, \ldots, 5 \right\rangle}$ and ${\left\langle K_md^mad^m,\ m=1, \ldots, 5 \right\rangle}$ are free on their given generators. Hence $H_3$ is a HNN-extension of $H_2$ along the isomorphism ${d^mad^mF_m^{\#} \mapsto K_md^mad^m},\ m=1, \ldots, 5$.
\end{proof}

\begin{lemma}\label{lem:G_tor_free}
The group $G$ is torsion-free.
\end{lemma}
\begin{proof}
Using Lemma~\ref{lem:H_are_HNN} we can apply Proposition~\ref{prop:hnn_torsion} to get the following equality chain:
\[ 
\torord(G) = \torord(H_{4})=\torord(H_{3})=\torord(H_{2})= \torord(F) = \emptyset
\]
\end{proof}

In the next section we will use this f.p.~`simulator' group $G$ to simulate equality in $\mathfrak{N}$, and hence equality in $A$.

\section{Using the simulator group $G$ to embed $A$ into a 13 generator 33 relation group $K_{4}$}\label{sec:simulation}

In this section we give an overview of the construction by Boone and Collins in p.$3$--$4$ in Section~I of \cite{ref:26_relations} (which is a variant of Aanderaa's proof of the Higman Embedding Theorem in \cite{ref:het_aand}) to obtain an explicit embedding of our original group $A$ into a group $K_{4}$ with $13$ generators and $33$ relations. Since $K_{4}$ is built up entirely by HNN-extensions, starting with the group $A*G$, we see that this embedding preserves torsion orders.

From Definition~\ref{def:H} we see that the group $G$ has finite presentation
$$s_1, s_2, q, t, k, a, d;$$
$$s_j^{-1}as_j = a \:\:\:\:\:\:\:\:\:\:\:\: s_j^{-1}ds_j = d^6ad^6$$
$$q^{-1}(d^mad^mF_m^{\#})q = K_md^mad^m$$
$$ta = at \:\:\:\:\:\:\:\:\:\:\:\: td = dt$$
$$ka = ak \:\:\:\:\:\:\:\:\:\:\:\: kd = dk$$
$$k^{-1}(\Phi_0^{-1}q^{-1}tq\Phi_0)k = (\Phi_0^{-1}q^{-1}tq\Phi_0)$$
for all ${m = 1,\ldots,5}$, ${j = 1,2}$, and where $\Phi_0 := \chi(p)s_2s_1^{2t} \:\:(= s_2s_1s_2s_1s_2^6s_1s_2^{r-3}s_1^{2t})$.

The following theorem, a combination of two results from \cite{ref:semigroup_to_group}, gives the key properties of $\mathfrak{N}_{*}$ and $G$ that we require.

\begin{theorem}\emph{\cite[Lemma~0 and Technical~Result~(i)]{ref:semigroup_to_group}}\label{thm:semigroup_to_simulator}
Let $\Delta$ and $\Pi$ be any words of $\mathfrak{N}$. Then 
\[
 \Delta \Pi = \Phi_0 \textnormal{ in } \mathfrak{N}\ \Leftrightarrow  \ \Delta q \Pi = q\Phi_0 \textnormal{ in } \mathfrak{N}_* \  \Leftrightarrow \ k^{-1}((\Delta^{\#} q \Pi)^{-1} t (\Delta^{\#} q \Pi))k = (\Delta^{\#} q \Pi)^{-1} t (\Delta^{\#} q \Pi) \textnormal{ in } G
\]

\end{theorem}
\begin{proof}
This is proved in \cite[p.$307$--$313$]{ref:semigroup_to_group}.
\end{proof}

Let ${\theta(W) := \chi(W)^{\#}}$. We will now show how we simulate equality in $A$ by equality in $G$.

\begin{theorem}\emph{\cite[Theorem~$2$]{ref:26_relations}}\label{thm:simulator_propterty}
Let $U$ be any word of $A$. Then 
\[
 U = 1 \textnormal{ in } A \ \Leftrightarrow \ k^{-1}((\theta(U) q \Phi_0)^{-1} t (\theta(U) q \Phi_0))k = (\theta(U) q \Phi_0)^{-1} t (\theta(U) q \Phi_0) \textnormal{ in } G
\]
\end{theorem}
\begin{proof}
This is just Theorem~\ref{thm:semigroup_properties}~(\ref{prt:simulation}) combined with Theorem~\ref{thm:semigroup_to_simulator}; see \cite[Theorem~2]{ref:26_relations}.
\end{proof}

Thus $G$ `simulates' equality in $A$. Note that $G$ has 7 generators and 14 relations, regardless of the original finite presentation of $A$ that we input. Note that, while $G$ `simulates' equality in $A$, there is no reason that $A$ should embed into $G$. Indeed, $G$ is always torsion-free (Lemma \ref{lem:G_tor_free}), regardless of whether $A$ has torsion or not. To embed $A$ into an f.p.~group with few generators and relators, we will now follow Aanderaa's proof of the Higman Embedding Theorem \cite{ref:het_aand}, in the same way as done on p.3--5 of \cite{ref:26_relations}.

\begin{defn}\label{def:K}
Let ${C = \left\langle c_1, c_2 \right\rangle}$ be an isomorphic copy of $A$. (Recall from the beginning of Section~$\ref{sec:group_to_semigroup}$ that w.l.o.g.~$A$ is $2$-generated.) Define ${k_0 := (q\Phi_0)k(q\Phi_0)^{-1}}$ and the following groups
\begin{align*}
K_1 &:= C*G \\
K_2 &:= \left\langle K_1;\: b_1, b_2 \:|\: b_i^{-1}s_jb_i = s_j,\: b_i^{-1}c_jb_i = c_j,\: b_i^{-1}k_0b_i = k_0c_i^{-1} \:\: (i,j=1,2) \right\rangle \\
K_3 &:= \left\langle K_2;\: f \:|\: f^{-1}\theta(a_i)^{\epsilon}b_i^{\epsilon}f = \theta(a_i)^{\epsilon},\: f^{-1}k_0f = k_0 \:\: (i=1,2, \: \epsilon = \pm 1) \right\rangle \\
K_4 &:= \left\langle K_3;\: h \:|\: h^{-1}th = tf, \text{\:} h^{-1}k_0h = k_0, \text{\:} h^{-1}s_jh = s_j \:\: (j=1,2) \right\rangle
\end{align*}
\end{defn}

\begin{lemma}\emph{\cite[Theorem~$3$~(i)]{ref:26_relations}}\label{lem:K_are_HNN}
$K_{i+1}$ is a HNN-extension of $K_{i}$ for ${i=1,2,3}$.
\end{lemma}
\begin{proof}
This is shown on p.$5$--$6$ of \cite{ref:26_relations}.
\end{proof}

Given how clear the construction of $K_{4}$ is, we can deduce its torsion orders.

\begin{lemma}\label{lem:K_torsion}
${\torord(K_4) = \torord(A)}$. 
\end{lemma}
\begin{proof}
Combining Lemmata~$\ref{lem:G_tor_free}$ and~$\ref{lem:K_are_HNN}$, with Propositions~$\ref{prop:hnn_torsion}$ and~$\ref{prop:amlg_torsion}$, gives that
\[
\torord(K_4) = \torord(K_3)=\torord(K_2)= \torord(K_1) = \torord(C)\cup\torord(G)  = \torord(A) \cup \emptyset
\]
\end{proof}

The reason for constructing $K_{4}$ was to embed $A$ into an f.p.~group with few generators and relations. We now see that $K_{4}$ can indeed be given by such a `small' finite presentation.

\begin{lemma}\emph{\cite[Theorem~$3$~(ii)]{ref:26_relations}}\label{lem:C_relations}
The defining relations of $C$ are redundant in $K_4$.
\end{lemma}
\begin{proof}
We will follow the proof given on p.$4$ of \cite{ref:26_relations}. Let $W_a$ be a word of $A$ with ${W_a = 1}$ in $A$. Let $W_b$ and $W_c$ be copies of $W_a$ with the $a_i$'s replaced with the corresponding $b_i$'s and $c_i$'s. To prove the result it is enough to show that ${W_c = 1}$ in $K_4$ without using any of the relations of $C$. Theorem~\ref{thm:simulator_propterty} tells us that
\begin{alignat}{2}
&& k^{-1}((\theta(W_a) q \Phi_0)^{-1} t (\theta(W_a) q \Phi_0))k &= (\theta(W_a) q \Phi_0)^{-1} t (\theta(W_a) q \Phi_0) \nonumber \\
&\Rightarrow\quad & k_0^{-1}\theta(W_a)^{-1} t \theta(W_a)k_0 &= \theta(W_a)^{-1} t \theta(W_a) \label{eqn:t_conj}
\end{alignat}
Now conjugate this by $h$ and using the relations introduced in $K_4$ we obtain 
\begin{alignat}{2}
&& h^{-1}k_0^{-1}\theta(W_a)^{-1} t \theta(W_a)k_0h &= h^{-1}\theta(W_a)^{-1} t \theta(W_a)h \nonumber \\
&\Rightarrow\quad & k_0^{-1}h^{-1}\theta(W_a)^{-1} t \theta(W_a)hk_0 &= h^{-1}\theta(W_a)^{-1} t \theta(W_a)h \nonumber \\
&\Rightarrow\quad & k_0^{-1}\theta(W_a)^{-1}h^{-1} t h\theta(W_a)k_0 &= \theta(W_a)^{-1}h^{-1} t h\theta(W_a) \nonumber \\
&\Rightarrow\quad & k_0^{-1}\theta(W_a)^{-1} tf \theta(W_a)k_0 &= \theta(W_a)^{-1} tf \theta(W_a) \label{eqn:tf_conj}
\end{alignat}
Now multiply the inverse of (\ref{eqn:t_conj}) by (\ref{eqn:tf_conj}) and then use the relations introduced in $K_3$ to get 
\begin{alignat*}{2}
&& k_0^{-1}\theta(W_a)^{-1} f \theta(W_a)k_0 &= \theta(W_a)^{-1} f \theta(W_a) \\
&\Rightarrow\quad & k_0^{-1} \theta(W_a)^{-1} \theta(W_a) W_b k_0 f &=  \theta(W_a)^{-1} \theta(W_a) W_b f \\
&\Rightarrow\quad & k_0^{-1} W_b k_0 &= W_b
\end{alignat*}
Now use the relations introduced in $K_2$ to get 
\begin{alignat*}{2}
&& W_b W_c &= W_b  \\
&\Rightarrow\quad & W_c &= 1 
\end{alignat*}
which is exactly what we wanted to show.
\end{proof}

So $K_{4}$ has a presentation with the 2 generators of $C$ (we discarded the relations of $C$), the generators (7) and relations (14) of $G$, and the stable letters (4) and HNN-relations (19) of the chain of three HNN-extensions from $C*G$ to $K_{4}$. Counting these up, we see that we have embedded $A$ into an f.p.~group $K_{4}$ with 13 generators and 33 relations, regardless of the input presentation for $A$.

\section{The explicit presentation of the group $K_{4}$ with 13 generators and 33 relations}\label{sec:explict}

We now give a (mostly) explicit presentation of $K_{4}$, and state its properties.

\begin{theorem}\label{thm:explicit_form}
There is a uniform algorithm for embedding a finitely presented group $A$ into another finitely presented group $K_4$ with $13$ generators and $33$ relations. Moreover, ${\torord(A) = \torord(K_4)}$, and $K_4$ is given by the presentation          
$$s_1, s_2, q, t, k, a, d, c_1, c_2, b_1, b_2, f, h;$$
$$s_j^{-1}as_j = a \:\:\:\:\:\:\:\:\:\:\:\: s_j^{-1}ds_j = d^6ad^6$$
$$q^{-1}(dads_2^{-1}s_1^{-2})q = s_1^2s_2s_1s_2dad$$ 
$$q^{-1}(d^2ad^2s_2^{-1}s_1^{-2})q = s_1^2s_2^2d^2ad^2$$
$$q^{-1}(d^3ad^3\tau(M)^{\#})q = \tau(N)d^3ad^3$$    
$$q^{-1}(d^{3+j}ad^{3+j}s_j^{-1})q = s_jd^{3+j}ad^{3+j}$$
$$ta = at \:\:\:\:\:\:\:\:\:\:\:\: td = dt$$
$$ka = ak \:\:\:\:\:\:\:\:\:\:\:\: kd = dk$$
$$k^{-1}(\Phi_0^{-1}q^{-1}tq\Phi_0)k = (\Phi_0^{-1}q^{-1}tq\Phi_0)$$
$$b_i^{-1}s_jb_i = s_j \:\:\:\:\:\:\:\:\:\:\:\: b_i^{-1}c_jb_i = c_j \:\:\:\:\:\:\:\:\:\:\:\: b_i^{-1}k_0b_i = k_0c_i^{-1}$$
$$f^{-1}\theta(a_i)^{\epsilon}b_i^{\epsilon}f = \theta(a_i)^{\epsilon} \:\:\:\:\:\:\:\:\:\:\:\: f^{-1}k_0f = k_0$$
$$h^{-1}th = tf \:\:\:\:\:\:\:\:\:\:\:\: h^{-1}k_0h = k_0 \:\:\:\:\:\:\:\:\:\:\:\: h^{-1}s_jh = s_j$$
for all ${i = 1,2}$, ${j = 1,2}$, ${\epsilon = \pm 1}$ and with the shorthand
\begin{align*}
\Phi_0 &:= \chi(p)s_2s_1^{2t} \:\:(= s_2s_1s_2s_1s_2^6s_1s_2^{r-3}s_1^{2t}) \\
k_0 &:= (q\Phi_0)k(q\Phi_0)^{-1}
\end{align*}
where $r$ and $t$ are the integers determined in Section~$3$.
\end{theorem}
\begin{proof}
By Lemma~\ref{lem:K_are_HNN} we have $C$ (which is an isomorphic copy of $A$) embedding into $K_4$. By Lemma~\ref{lem:K_torsion} we have ${\torord(A) = \torord(K_4)}$. Finally by Lemma~\ref{lem:C_relations} the defining relations of $C$ in $K_4$ are redundant and so $K_4$ has the given presentation.
\end{proof}

We can apply this construction to certain universal groups to get new ones with few generators and relations. 
\begin{theorem}\label{thm:lowish_relator_uni_group}
Let $X$ be a recursively enumerable set. Then there is a universal finitely presented $X$-torsion-free group with $13$ generators and $33$ relations with a presentation as given in Theorem~$\ref{thm:explicit_form}$.
\end{theorem}
\begin{proof}
Apply Theorem~\ref{thm:explicit_form} to any universal finitely presented $X$-torsion-free group, for example the one from Theorem \ref{thm:uni_X-tor_free}.
\end{proof}

\begin{remark}
We note that although most of the relations of $K_4$ are determined explicitly, the relation ${q^{-1}(d^3ad^3\tau(M)^{\#})q = \tau(N)d^3ad^3}$ is extremely long and depends entirely on our initial choice of finite presentation for $A$. As such we cannot give a completely explicit presentation for these universal groups using this method. It would be interesting to see a construction of a completely explicit presentation for a universal f.p.~torsion-free group. We have not been able to do this, as the Turing machine needed in Theorem \ref{thm:uni_X-tor_free} (coming from \cite[Theorem 2.11]{ref:uni_X-tor_free}) appears to be very difficult to construct. Theoretically this should be possible by the Church-Turing Thesis.
\end{remark}

\section{Embedding an f.p.~group with several commuting generators into one with fewer generators and fewer relations}\label{sec:borisov}

In  this section we give an overview of the construction by Collins in \cite{ref:borisov_col} (which is itself a refinement of an idea due to Borisov in \cite{ref:borisov_orig}) for embedding an f.p.~group with many commuting generators into an f.p.~group with fewer generators and relations. Borisov's original construction dealt with the case where one generator $c$ commutes with a set of  generators $\{u_{j}\}$. Collins' refinement deals with the more general case, where one set of generators $\{c_{j}\}$ simultaneously commute with another set $\{u_{k}\}$ (that is, $c_{j}u_{k}=u_{k}c_{j}$ for all $j,k$).

We note that this embedding preserves torsion orders, as it is built from HNN-extensions and an amalgamated free product. Due to the length of many of the relations, an explicit form will not be given.

\begin{defn}\label{def:gamma}
Suppose we have a finite presentation of a group
\begin{equation*}
\Gamma = \left\langle x_i,\: c,\: u_j \:|\: D,\: cu_j = u_jc \right\rangle
\end{equation*}
with $D$ a set of relations. We construct the following groups:
\begin{align*}
\Gamma_1 &:= \left\langle \Gamma;\: a,\: d |\: a^{-1}ca = c,\: d^{-1}cd = c \right\rangle \\
C'       &:= \left\langle c' \right\rangle &(\text{an isomorphic copy of } C = \left\langle c \right\rangle) \\
K        &:= \left\langle C';\: x,\: y \:|\: x^{-1}c'x = c',\: y^{-1}c'y = c' \right\rangle \\
\Gamma_2 &:= \left\langle \Gamma_1*K \:|\: c = c',\: d = y,\: u_ja^{-j}da^j = x^{-j}yx^j \right\rangle \\
\Gamma_3 &:= \left\langle \Gamma_2;\: b \:|\: b^{-1}cb = c,\: b^{-1}ab = d,\: b^{-1}db = x \right\rangle
\end{align*}
\end{defn}

\begin{lemma}\label{lem:gamma_HNN}
We have the following extensions
\begin{enumerate}[(i)]
\item $\Gamma_1$ is an HNN-extension of $\Gamma$
\item $K$ is an HNN-extension of $C'$
\item $\Gamma_2$ is an amalgamated free product of $\Gamma_1$ and $K$
\item $\Gamma_3$ is an HNN-extension of $\Gamma_2$
\end{enumerate}
\end{lemma}
\begin{proof}
These are shown in the proof of \cite[Theorem]{ref:borisov_col}.
\end{proof}

The clear construction of $\Gamma_{3}$ from $\Gamma$ allows us to understand the torsion orders of $\Gamma_{3}$.

\begin{lemma}\label{lem:gamma_torsion}
${\torord(\Gamma_3) = \torord(\Gamma)}$.
\end{lemma}
\begin{proof}
As ${C \leqslant \Gamma}$ we have ${\torord(C) \subseteq \torord(\Gamma)}$. So Lemma~${\ref{lem:gamma_HNN}}$ now implies that $\torord(\Gamma_3)$ $=\torord(\Gamma)\cup\torord(C) = \torord(\Gamma)$, via Propositions~\ref{prop:hnn_torsion}~and~\ref{prop:amlg_torsion}.
\end{proof}

Borisov \cite{ref:borisov_orig} used the above construction to embed $\Gamma$ into an f.p.~group with fewer generators and relators; we follow the exposition of this as given by Collins in \cite{ref:borisov_col}.

\begin{theorem}\emph{\cite[Theorem~1]{ref:borisov_col}}\emph{(Borisov's Theorem)}\label{thm:borisov_thm}
Define the map on words induced by ${\lambda(c) := c}$, ${\lambda(x_i) = x_i}$ and 
\begin{equation*}
\lambda(u_j) := b^{-2}a^{-j}bab^{-1}a^jb^2a^{-j}b^{-1}a^{-1}ba^j
\end{equation*}
Define the group 
\begin{equation*}
\Gamma' := \left\langle x_i,\: a,\: b,\: c \:|\: \lambda(D),\: ac = ca,\: bc = cb \right\rangle
\end{equation*}
Then the induced map $\overline{\lambda} : \Gamma \rightarrow \Gamma'$ is an embedding of groups with $\torord(\Gamma') = \torord(\Gamma)$.
\end{theorem}
\begin{proof}
We can remove excess generators and relations to see that $\Gamma_3$ is the same group as $\Gamma'$ and that we get the given embedding $\overline{\lambda}$. 
\end{proof}

Collins \cite{ref:borisov_col} was able to refine Borisov's construction, to deal with the case where there are two sets of commuting generators.

\begin{theorem}\emph{\cite[Theorem~2]{ref:borisov_col}}\label{thm:borisov_thm_var}
Suppose we have a finite group presentation
\begin{equation*}
\Gamma = \left\langle x_i,\: c_j,\: u_k \:|\: D,\: c_ju_k = u_kc_j \right\rangle
\end{equation*}
Then there is a map $\lambda$ (which can be found algorithmically) such that $\Gamma$ embeds into
\begin{equation*}
\Gamma' := \left\langle x_i,\: a,\: b,\: p,\: q \:|\: \lambda(D),\: ap = pa,\: bp = pb,\: aq = qa,\: bq = qb \right\rangle
\end{equation*}
with ${\torord(\Gamma') = \torord(\Gamma)}$. 
\end{theorem}

Moreover, by setting $J=|\{c_{j}\}|$ and $K=|\{u_{k}\}|$, we see that $\Gamma'$ has $J+K-4$ fewer generators, and $JK-4$ fewer relations, than $\Gamma$ (assuming $J+K\geq 4$ and $JK\geq 4$).

\begin{proof}
This is similar to the above, see \cite[Theorem 2]{ref:borisov_col} for details of the changes required. The construction is still formed entirely of HNN-extensions and amalgamated free products, so we still have ${\torord(\Gamma') = \torord(\Gamma)}$.
\end{proof}

Although Theorems~\ref{thm:borisov_thm} and \ref{thm:borisov_thm_var} are embedding theorems, they should be seen as a sort of clever rewriting process.

\section{Embedding $K_{4}$ into a group with 8 generators and 26 relations}\label{sec:final}

We can now use Borisov's Theorem, and its variant by Collins, to prove our main results.

\begin{theorem}\label{thm:main_result}
There is a uniform algorithm for embedding a finitely presented group $A$ into another finitely presented group $H$ with $8$ generators and $26$ relations. Moreover, ${\torord(A) = \torord(H)}$. 
\end{theorem}
\begin{proof}
In $K_4$ replace the generators $a$, $d$ and $k$ with ${a_0 := q\Phi_0 a \Phi_0^{-1}q^{-1}}$, ${d_0 := q\Phi_0 d \Phi_0^{-1}q^{-1}}$ and ${k_0 := q\Phi_0 k \Phi_0^{-1}q^{-1}}$. We now have relations for $k_0$ commuting with $a_0$, $d_0$, $f$, $h$ and $t$. So applying Theorem~\ref{thm:borisov_thm} we can reduce the number of generators by $3$ and and the number of relations by $3$. Also we have relations for $b_1$ and $b_2$ commuting with $s_1$, $s_2$, $c_1$ and $c_2$. So we can apply Theorem~\ref{thm:borisov_thm_var} to reduce the number of generators by $2$ and the number of relations by $4$.
\end{proof}

\begin{theorem}\label{thm:low_relator_uni_group}
Let $X$ be a recursively enumerable set. Then there is a universal finitely presented $X$-torsion-free group with $8$ generators and $26$ relations.
\end{theorem}
\begin{proof}
Apply Theorem~\ref{thm:main_result} to any universal finitely presented $X$-torsion-free group, for example the one from Theorem \ref{thm:uni_X-tor_free}.
\end{proof}

\ 

\noindent \scriptsize{\textsc{Department of Pure Mathematics and Mathematical Statistics, University of Cambridge,
\\Wilberforce Road, Cambridge, CB3 0WB, UK. 
\\mcc56@cam.ac.uk
\vspace{5pt}
\\Selwyn College, Cambridge
\\Grange Road, Cambridge, CB3 9DQ, UK.
\\meh69@cam.ac.uk}

\end{document}